\documentclass[11pt]{amsart}
\usepackage{graphicx,amssymb,amsmath, latexsym,cite, amscd,stmaryrd}
\usepackage[foot]{amsaddr}  
\usepackage{amsthm,amsrefs} 
\usepackage{enumerate,lscape}
\usepackage{float,xcolor}
\usepackage{fancyhdr, lastpage}
\usepackage{libertine}
\usepackage{url}  
\usepackage{verbatim,lineno}
  \usepackage{arydshln}
\usepackage{tikz-cd}
\usepackage{mathtools}
\usepackage{amsfonts}

\newtheoremstyle{mythm}                   
{6pt}
{6pt}
{\it}
{}
{\bf}
{.}
{.5em}
{}

\newtheoremstyle{mydef}                   
{6pt}
{6pt}
{}
{}
{\bf}
{.}
{.5em}
{}

\newtheoremstyle{myrem}                   
{6pt}
{6pt}
{}
{}
{\bf}
{.}
{.5em}
{}

\theoremstyle{mythm}
\newtheorem{theorem}{Theorem}[section]
\newtheorem{prop}[theorem]{Proposition}
\newtheorem{lemma}[theorem]{Lemma}

\theoremstyle{mydef}

\theoremstyle{myrem}

\numberwithin{equation}{section}

\reversemarginpar
\subjclass[2000]{}

\newcommand{\R}{\mathbb{R}}
\newcommand{\C}{\mathbb{C}}

\newcommand{\Aut}{\mathrm{\mathop{Aut}}}
\newcommand{\re}{\mathrm{re}}
\newcommand{\imm}{\mathrm{im}}

\newcommand{\ad}{\mathrm{\mathop{ad}}}
\newcommand{\Ad}{\mathrm{\mathop{Ad}}}
\newcommand{\GL}{\mathrm{\mathop{GL}}}
\newcommand{\SL}{\mathrm{\mathop{SL}}}
\newcommand{\SU}{\mathrm{\mathop{SU}}}

\newcommand{\g}{\mathfrak{g}}
\newcommand{\fk}{\mathfrak{k}}
\newcommand{\fp}{\mathfrak{p}}

\renewcommand{\leq}{\leqslant}

\renewcommand{\geq}{\geqslant}

\newcommand{\gl}{\mathfrak{gl}}
\renewcommand{\sl}{\mathfrak{sl}}

\newcommand{\h}{\mathfrak{h}}
\newcommand{\rre}{\rho^{\mathrm{re}}}
\newcommand{\rim}{\rho^{\mathrm{im}}}

\newcounter{ithmcount}
\newenvironment{iprf}{\begin{list}{{\rm
	\alph{ithmcount})}}{\usecounter{ithmcount}\labelwidth-5pt
      \leftmargin0pt \topsep3pt \itemsep1pt \parsep2pt}}{\qedhere\end{list}}

\newenvironment{ithm}{\begin{list}{{\rm \alph{ithmcount})}}{\usecounter{ithmcount}\labelwidth18pt
      \leftmargin18pt \topsep3pt \itemsep1pt \parsep2pt}}{\end{list}}

\begin{document}

\vspace*{-2cm}

\title[]{Computing the real Weyl group}
\author[H. Dietrich]{Heiko Dietrich}
\address{School of Mathematics, Monash University,
  Clayton VIC 3800, Australia}
\email{heiko.dietrich@monash.edu}
\author[W.\ A.\ de Graaf]{Willem A.\ de Graaf} 
\address{Department of Mathematics, University of Trento, Povo (Trento), Italy}
\email{degraaf@science.unitn.it}

\keywords{real Weyl group, real semisimple Lie algebra, computational Lie theory}
\date{\today}
\thanks{This work is  supported by an Australian Research Council grant, identifier DP190100317.}


\begin{abstract}
Let $\g$ be a semisimple Lie algebra over the real numbers. We describe an explicit combinatorial construction of the real Weyl group of $\g$ with respect to a given Cartan subalgebra. An efficient computation of this  Weyl group is important for the classification of regular semisimple subalgebras, real carrier algebras, and real nilpotent orbits associated with $\g$; the latter have various applications in theoretical physics.
\end{abstract}

\maketitle%

\section{Introduction}
\noindent Let $\g^c$ be a semisimple Lie algebra over the complex numbers. The adjoint group $G^c$ of $\g^c$ is the identity component (in the Zariski topology) of the automorphism group of $\g^c$. Up to conjugacy in $G^c$, there is a unique Cartan subalgebra $\h^c\leq \g^c$; let $\Phi$ be the corresponding root system. The reflections defined by all those roots generate the Weyl group $W(\Phi)$ of $\g^c$; the latter can also be defined as  $W(\Phi)=N_{G^c}(\h^c)/Z_{G^c}(\h^c)$. Root systems and Weyl groups are arguably the most important tools in Lie theory, because many problems can be reduced to computations with those combinatorial objects.  The situation is similar, but more complicated when considering Lie algebras over the real numbers; one common issue is that now Cartan subalgebras and Weyl groups are not necessarily unique. We refer to the books of Humphreys \cite{humph}, Knapp \cite{knapp}, or  Onishchik \cite{onish} for extensive background information.

Now consider a real form $\g$ of $\g^c$, with associated complex conjugation~$\sigma$, and choose a Cartan involution $\theta$ of $\g$. Let $G$ be the subgroup of $G^c$ consisting of all $g\in G^c$ with $g(\g)=\g$, and write  $G^\circ$ for its identity component (in the Euclidean
topology). Following \cite[(7.92a)]{knapp}, we define the real Weyl group of $\g$ with respect to a $\theta$-stable Cartan subalgebra $\h$ of $\g$ as \[W(\g,\h)=N_{G^\circ}(\h)/Z_{G^\circ}(\h).\]
In our paper \cite{dfg2}, an algorithm is given to compute regular semisimple subalgebras of $\g$ up to
conjugacy by $G^\circ$. In this algorithm the real Weyl group plays a paramount role. Similarly, the real Weyl
group can be used to classify carrier subalgebras and nilpotent orbits associated with a grading of $\g$,
see \cite{dfg2} for details. Related to that, in \cite{dgrt} an application is discussed of nilpotent orbit
classifications to theoretical physics (supergravity). It is therefore of interest to understand the structure of $W(\g,\h)$ and to have computational tools that can be used to construct it.

In a more general context, the ATLAS project \cite{atlas} considered $W(G',H)=N_{G'}(H)/Z_{G'}(H)$ for an arbitrary real form $G'$ of $G^c$ with $\theta$-stable Cartan subgroup $H\leq G'$, see  \cite{vogan,adams,dc2}. It is shown in \cite[Propositions 4.11 \& 4.16]{vogan} that 
\[W(G',H)\cong (W^{\mathrm{c}})^\theta\ltimes (W^{\re}\times (W(G',H)\cap W^{\imm})\]
with $W(G',H)\cap W^{\imm}=A\ltimes W^{\imm,\mathrm{c}};$ here $W^\re$, $W^\imm$, $W^{\imm,\mathrm{c}}$ are the Weyl groups of the root systems of $\g^c$ consisting of real, imaginary, and compact imaginary roots, respectively; moreover, $W^{\imm}=Q\ltimes W^{\imm,\mathrm{c}}$ for some elementary abelian 2-subgroup $Q$ containing $A$. More details and the definition of $(W^{\mathrm{c}})^\theta$ are given in Section \ref{secRSS}. While each of $(W^{\mathrm{c}})^\theta$, $W^{\imm,\mathrm{c}}$, $W^{\imm}$, and $Q$ can be computed from the Lie algebra data alone, the construction of $A$ is the complicated part and depends on the isogeny type of the real Lie group $G'$. It is \cite[Corollary 6.10]{dc2} that gives a clue for this construction. 

The aim of this paper is to describe  how to construct $W(\g,\h)$ by computer. 
Because we focus on the group $G^\circ$, which is completely determined by $\g$, it is in principle possible
to determine $W(\g,\h)$ using only information from $\g$. Here we provide efficient algorithms to do this.
For this we give a self-contained and detailed proof of the decomposition \begin{eqnarray}\label{eqWDEC}W(\g,\h)=(W^{\mathrm{c}})^\theta \ltimes (W^{\mathrm{re}}\times (A\ltimes W^{\imm,\mathrm{c}}))
\end{eqnarray}
and explain how all the subgroups involved can be computed.
While guided by the proofs in \cite{vogan,adams,dc2}, our description attempts to be largely self-contained and to avoid, as much as possible, the technical details in those papers. The latter is achieved by specialising results to $G^\circ$, and by rewriting some proofs of \cite{vogan,adams,dc2} assuming not much more than basic properties of root systems.
Our implementation of our algorithm is contained in the software package CoReLG for the system GAP \cite{gap}. 

The structure of the paper is as follows. In Section \ref{secNot} we introduce the notation used in this work. In Section \ref{secRSS} we discuss various root sub-systems and their Weyl groups. Our set-up allows us to describe in detail the proof of the construction of the subgroup $A$, see Proposition \ref{propWI}. In turn, this allows us to prove the main result, a combinatorial construction of  $W(\g,\h)$, in Theorem \ref{thmFINAL}. We conclude with some examples in Section \ref{secComp}.

We note that our paper \cite{dfg2} also comments on the construction of  $N_{G'}(\h)/Z_{G'}(\h)$, but some confusing assumptions have been posed on $G'$, namely that $G'=G^c(\R)$ is a group of real points and also connected; see also the clarification in \cite[Remark 10]{dgrt}.  Here we consider the group $G'=(G^c(\R))^\circ$, which is called the adjoint group of $\g$, see \cite[Section II.5]{helga}.


\section{Notation}\label{secNot}
\noindent We use basic knowledge on root systems, such as bases of simple roots, positive roots, Weyl groups; for background information on Lie algebras and root systems we refer to standard books, such as Humphreys \cite{humph}, Knapp \cite{knapp}, or  Onishchik \cite{onish}. Throughout, we use the following notation. Let $\g^c$ be a semisimple complex Lie algebra with real form $\g$ and associated conjugation $\sigma$, that is, $\sigma( x+\imath y ) = x-\imath y$ for $x,y\in \g$. Let  $\theta$
be a Cartan involution of $\g$ with Cartan decomposition $\g=\fk\oplus\fp$, so that  $\tau=\sigma\circ\theta$ is a compact
structure on~$\g^c$. We refer to \cite[Section 2]{dfg} for details on the construction of real forms. Let $G^c$ be the adjoint group of $\g^c$, which can be defined as the identity component
(in the Zariski topology) of the automorphism group of $\g^c$, see \cite[(I.7)]{onish}. Let $G$ be
the group consisting of all $g\in G$ with $g(\g)=\g$, that is, $g\in G^c$ lies in $G$ if and only if $g\circ \sigma = \sigma\circ g$.
Let $G^\circ$ denote the identity component of $G$ in the Euclidean  
topology. Note that $G^c$ is an algebraic group and its Lie algebra  is spanned by $\ad_\g x$ where  $x$ runs over the elements of a basis of $\g$, see \cite[(I.1)]{onish}; since every such basis is defined over $\R$, the group $G^c$ is defined over $\R$. In particular, we can view $G^c$  as a
matrix group by taking matrices with respect to a fixed basis of $\g$; in this situation,  it follows that $G=G^c(\mathbb{R})$ is the group of real points of~$G^c$. We fix a $\theta$-stable Cartan subalgebra $\h$ of $\g$ and let $\h^c$ be its
complexification; note that $\h^c$ is a Cartan subalgebra of $\g^c$. We define \[W(\g,\h)=N_{G^\circ}(\h)/Z_{G^\circ}(\h)\] as the real Weyl group of $\g$ with respect to $\h$. Let $\Phi$ be
the root system of $\g^c$ with respect to $\h^c$. The abstract Weyl group defined by $\Phi$ is denoted by  $W=W(\Phi)$ and generated by all reflections $s_\alpha$ where $\alpha$ runs over a set of simple roots of $\Phi$, see \cite[Section 9.3]{humph}. It is well known (see for example \cite[(I.7) \& (II.16)]{onish}) that one can also define $W$ analytically as
\[W=N_{G^c}(\h^c)/Z_{G^c}(\h^c).\]
Fix a Chevalley basis of $\g^c$, consisting of semisimple elements $h_1,\ldots,h_\ell\in \h^c$ and root vectors $x_\alpha$ for $\alpha\in \Phi$, see \cite[Section~25.2]{humph}. For $h\in \h^c$ and $\alpha\in \Phi$ we have $[h,x_\alpha]=\alpha(h)x_\alpha$, hence
$$[h,\theta(x_\alpha)] = \theta ( [\theta(h),x_\alpha] ) =
 \alpha(\theta(h)) \theta(x_\alpha),$$
and so $\theta(x_\alpha)$ lies in the root space corresponding to $\alpha\circ
\theta$. In particular, $\alpha\circ\theta\in\Phi$, so we have an involution
$\alpha \mapsto \alpha\circ\theta$ of $\Phi$;  we write
$\theta(\alpha)$ for $\alpha\circ\theta$, and we extend this involution
to the dual space $(\h^c)^*$. Let $\kappa$ denote the Killing form of $\g^c$; since $\g^c$ is semisimple, $\kappa$ is a symmetric, non-degenerate bilinear form. Its restriction to
$\h^c$ gives the well-known bijection $(\h^c)^* \to \h^c$, $\mu\mapsto
h_\mu'$, where $h_\mu'$ is defined by $\mu(-) = \kappa( -, h_\mu')$. For
$\mu,\lambda
\in (\h^c)^*$ define $(\mu,\lambda) = \kappa( h_\mu', h_\lambda')$,
and for $\alpha,\beta\in \Phi$ write $\langle \alpha,\beta^\vee\rangle=2(\alpha,\beta)/(\beta,\beta)$. We conclude with an observation.

\begin{lemma}\label{lemForm} The restriction of $(-,-)$ to the real span of $\Phi$ is a $\theta$-invariant inner product.
\end{lemma}
\begin{proof}
It is well-known that this restriction is an inner product, see \cite[Section 8.5]{humph}. Note that $\ad_{\g^c}\gamma(x)=\gamma\circ\ad_{\g^c}(x)\circ \gamma^{-1}$ for every $x\in\g^c$ and automorphism $\gamma\in\Aut(\g^c)$, which shows  that $\kappa(\theta(x),\theta(y)) = \kappa(x,y)$ for all $x,y\in \g^c$. This implies that $\kappa(h,\theta(h_\mu')) = \kappa(\theta(h),
h_\mu')= \mu(\theta(h))$,  hence $h_{\theta(\mu)'} = \theta(h_\mu')$. This shows that $(\theta(\mu),\theta(\lambda)) = (\mu,\lambda)$ for all $\mu,\lambda\in
(\h^c)^*$, as claimed.
\end{proof}

\section{Some root subsystems}\label{secRSS}
\noindent As indicated in the introduction, the decomposition \eqref{eqWDEC} of $W(\g,\h)$ is induced by several sub-root systems of $\Phi$; we introduce and discuss those sub-root systems here. 

Let $\Psi$ be a root system and recall that we write $W(\Psi)$ for its Weyl group. A subset  $\Pi\subset \Psi$ is a sub-root system (or just subsystem) if
for $\alpha,\beta\in \Pi$ we have $-\alpha\in \Pi$ and, if $\alpha+\beta\in \Psi$, then $\alpha+\beta\in \Pi$.  For a given system of positive roots $\Psi^+$ the corresponding Weyl vector is
\[\rho(\Psi) = \tfrac{1}{2}\sum\nolimits_{\alpha\in \Psi^+} \alpha.\]
Note that $(\rho(\Psi),\alpha)=1$ for simple roots $\alpha\in\Psi$, see the proof of \cite[Lemma 13.3A]{humph}, which implies:
\begin{lemma}\label{lemTriv}
Let $\Psi$ be a root system. If $\alpha\in\Psi^+$, then $(\rho(\Psi),\alpha)>0$.
\end{lemma}

Recall that $\Phi$ is the root system of $\g^c$ with respect to $\h^c$. The subsystems of real, imaginary, compact imaginary, and noncompact imaginary roots are defined as 
\begin{align*}
\Phi^{\mathrm{re}} &= \{ \alpha\in \Phi \mid \theta(\alpha) = -\alpha\},&
\Phi^{\mathrm{im}} &= \{ \alpha\in \Phi \mid \theta(\alpha) = \alpha\},\\
\Phi^{\imm,\mathrm{c}} &= \{ \alpha\in \Phi^\imm \mid \theta(x_\alpha) = x_\alpha\},&
\Phi^{\imm,\mathrm{nc}} &= \{ \alpha\in \Phi^\imm \mid \theta(x_\alpha) = -x_\alpha\}
\end{align*}with $\Phi^\imm=\Phi^{\imm,\mathrm{c}}\cup \Phi^{\imm,\mathrm{nc}}$, see \cite[p.\ 390]{knapp}. We set $\rho^{\mathrm{re}} = \rho(\Phi^{\mathrm{re}})$ and $\rho^{\mathrm{im}} =
\rho(\Phi^{\mathrm{im}})$, and define
\[\Phi^{\mathrm{c}} = \{ \alpha\in \Phi \mid (\alpha,\rho^{\mathrm{re}})=
(\alpha,\rho^{\mathrm{im}})=0\}.\]
Note that $\Phi^{\mathrm{c}}$ is  $\theta$-invariant as $\theta(\rim)=\rim$ and
$\theta(\rre)=-\rre$. Let $\alpha$ be a positive root; Lemma~\ref{lemTriv} shows that if $\alpha$ is imaginary, then $(\alpha,\rim)\neq 0$, and if  $\alpha$ is real, then $(\alpha, \rre)\neq 0$. This proves that $\Phi^{\mathrm{c}} \cap \Phi^{\mathrm{im}} = \Phi^{\mathrm{c}} \cap \Phi^{\mathrm{re}}=\emptyset$. In the following we denote the Weyl group of $\Phi^X$  by $W^X$ for labels $X=\mathrm{c},\imm,\re,\ldots$. Note that $\theta$ acts on $\Phi$, so if $\Psi$ is a $\theta$-invariant subsystem of $\Phi$, then we denote the fixed points of $\theta$ in $W(\Psi)$ by  \[W(\Psi)^\theta=\{w\in W(\Psi)\mid \theta\circ w=w\circ\theta\}.\]

\subsection{Preliminary results} We need three preliminary results; the first comes from \cite[Lemma~3.1]{vogan}. We include an expanded proof that gives full details using basic results about root systems.

\begin{lemma}\label{lem:1}
Let $\Psi$ be a $\theta$-invariant subsystem of $\Phi$ with $\Psi\cap \Phi^{\mathrm{re}} = \Psi\cap \Phi^{\mathrm{im}}=\emptyset$. Then $\Psi$ is the disjoint orthogonal union of two subsystems, that is, $\Psi=\Psi_1\cup\Psi_2$ with $\Psi_1\cap\Psi_2=\emptyset$ and  $(\alpha,\beta)=0$ for all $\alpha\in \Psi_1$ and $\beta\in \Psi_2$. Furthermore, $\theta \colon \Psi_1 \to \Psi_2$ is an
isomorphism and\[W(\Psi) = W(\Psi_1)\times W(\Psi_2)\quad\text{and}\quad W(\Psi)^\theta = \{ (w,\theta w\theta) \mid w\in W(\Psi_1)\} \cong W(\Psi_1).\]
In particular, $W(\Psi)^\theta$ is generated by $s_\alpha s_{\theta(\alpha)}$ where $\alpha$ runs over a set of simple roots of $\Psi_1$.
\end{lemma}
 
\begin{proof}
The real space $V$ spanned by $\Psi$ is $\theta$-invariant, so we have an eigenspace decomposition $V=V_1\oplus V_{-1}$ is .   Because of the $\theta$-invariance of the bilinear form, $V_1$ and $V_{-1}$ are  orthogonal. If $V_1=0$, then all roots of $\Psi$ are real, hence $\Psi=\Psi\cap \Phi^\re=\emptyset$, which is not possible; thus $V_1\ne 0$.  If $\alpha\in \Psi$ satisfies $(\alpha,v)=0$ for all $v\in V_1$, then $\alpha\in V_{-1}$, so $\alpha\in\Psi\cap \Phi^\re=\emptyset$, which is not possible; thus, there is no $\alpha\in \Psi$ which is orthogonal to $V_1$.  It follows that there is a $v_0\in V_1$ such that $(v_0,\alpha)\neq 0$ for all $\alpha\in \Psi$: simply choose $v_0$ outside the hyperplanes defined by each $\alpha\in\Psi$. We use this vector to define a positive system $\Psi^+ = \{ \alpha\in \Psi \mid (\alpha,v_0)>0\}$ for $\Psi$: indeed, setting $\alpha < \beta$ if and only if  $(\alpha,v_0)  <(\beta,v_0)$ defines  a root order, see \cite[p.\ 164]{gra6}. Note that $\theta( \Psi^+) = \Psi^+$ since $\theta(v_0)=v_0$. By \cite[Sections 10.4 \& 11.3]{humph} we can decompose $\Psi = \Pi_1 \cup \cdots \cup \Pi_m$, where the $\Pi_i$ are uniquely determined irreducible subsystems that are pairwise orthogonal; since $\Psi$ is $\theta$-invariant, $\theta$ permutes these subsystems. Suppose  $\theta(\Pi_i) = \Pi_i$ for some $i$. Then $\Pi_i^+ = \Pi_i\cap \Psi^+$ is a positive system for $\Pi_i$. The highest root $\beta_i$ of $\Pi_i^+$ is  uniquely
  determined, see \cite[Lemma 10.4A]{humph}, and because $\theta$ fixes $\Pi_i^+$ we have $\theta(\beta_i)=\beta_i$. But then $\beta_i \in \Psi\cap \Phi^{\mathrm{im}}=\emptyset$, which is impossible; this shows that $\theta(\Pi_i)\ne \Pi_i$ for all $i$. In particular, we can partition $\Psi=\Psi_1\cup\Psi_2$ such that $\theta(\Psi_1)=\Psi_2$. Now most of the statements of the lemma now follow. If $w=(w_1,w_2)\in W(\Psi)^\theta$, then for $\alpha\in \Psi_1$  we have $w(\theta(\alpha)) = w_2(\theta (\alpha))$ and $w(\theta(\alpha))=\theta(w(\alpha)) = \theta w_1(\alpha)$; hence $w_2(\theta(\alpha)) = \theta(w_1 \theta (\theta(\alpha)))$ for all $\alpha\in\Psi_1$, which shows that $w_2 = \theta w_1\theta$. For the last statement note
  that $\theta s_\alpha \theta = s_{\theta(\alpha)}$. \end{proof}

\begin{lemma}\label{lem:rts}
  Let $\Psi$ be a root system with fixed basis of simple roots $\Delta$ and real span  $V$. Let $v\in V$ be such that $(\alpha,v)\geq 0$ for
  all $\alpha\in \Delta$. Let $\Psi^v$ be the subsystem  $\{ \alpha\in \Psi \mid (\alpha,v)=0\}$, and set $W^v =\{ w\in W(\Psi) \mid w(v)=v\}$. Then $\Psi^v\cap\Delta$ is a basis of
  simple roots of $\Psi^v$ and $W(\Psi^v) = W^v$.
\end{lemma}

\begin{proof}
  This proof follows standard ideas in Lie theory, see for example,   \cite[Lemma 10.3B]{humph} and   \cite[Lemma 8.3.4]{gra6}.  
Write $\Delta = \{\alpha_1,\ldots,\alpha_m\}$ and $s_i=s_{\alpha_i}$ for each $i$. Every $\alpha\in \Psi^v\cap \Psi^+$ can be written as  $\alpha = \sum_i k_i \alpha_i$ with integers $k_1,\ldots,k_m\geq 0$, see \cite[Section 10.1]{humph}. By assumption $0=(\alpha,v) = \sum_i k_i (\alpha_i,v)$, so $k_i$ can only be nonzero if $\alpha_i \in \Psi^v\cap\Delta$.  This proves
the first part.

If $\alpha\in \Psi^v$, then $s_\alpha(v)=v$, and we see that $W(\Psi^v)\subset W^v$.  Let $w\in W(\Psi)$  with reduced expression  $w=s_{i_1}\cdots s_{i_t}$ for some $t>0$. For $1\leq j\leq t+1$ set
$v_j = s_{i_j}\cdots s_{i_t}(v)$ with  $v_{t+1}=v$, so that
$$(v_j,\alpha_{i_{j-1}}) = (s_{i_j}\cdots s_{i_t}(v),\alpha_{i_{j-1}}) =(v,s_{i_t}\cdots s_{i_j}(\alpha_{i_{j-1}}))$$for all $2\leq j\leq t+1$. Each $\alpha_i$ permutes the positive roots other than $\alpha_i$, and $s_i(\alpha_i)=-\alpha_i$, see \cite[Lemma 10.2.B]{humph}. Since $s_{i_t}\cdots s_{i_j}s_{i_{j-1}}(\alpha_{i_{j-1}})$
is a negative root, see \cite[Corollary 8.3.3]{gra6}, the root $s_{i_t}\cdots s_{i_j}(\alpha_{i_{j-1}})$ is positive; now $(v_j,\alpha_{i_{j-1}})\geq 0$ by the assumption on $v$. Thus, $v_{j-1} = s_{i_{j-1}}(v_j) =v_j-a_{j-1}\alpha_{i_{j-1}}$ where $a_{j-1} =2(v_j,\alpha_{i_{j-1}})/(\alpha_{i_{j-1}},\alpha_{i_{j-1}})\geq 0$. 
So $w(v)=v$ if and only if $v=v-a_1\alpha_{i_{j_1}}-\ldots-a_t\alpha_{i_{j_t}}$, if and only if each $a_i=0$, if and only if each $v_j=v$ and $(v,\alpha_{i_j})=0$.
\end{proof}

The following result is attributed to Chevalley in  \cite[Proposition 3.8]{vogan}; since we could not find a proof in the literature, we include it here.

\begin{lemma}\label{lem:2}
  Let $\Psi$ be a root system contained in a real  space $V$, with Weyl group $W=W(\Psi)$.
For given $\lambda_1,\ldots,\lambda_m\in V$ define \begin{eqnarray*} W^{\lambda_1,\ldots,\lambda_m}& =& \{ w\in W \mid  w(\lambda_i)=\lambda_i\text{ for all $i$}\}\quad{\text{and}}\\ \Psi^{\lambda_1,\ldots,\lambda_m} &=& \{ \alpha\in \Psi\mid  (\alpha,\lambda_i) = 0\text{ for all $i$}\}.
\end{eqnarray*}
Then $\Psi^{\lambda_1,\ldots,\lambda_m}$ is a
  subsystem of $\Psi$ with Weyl group $W(\Psi^{\lambda_1,\ldots,\lambda_m}) = W^{\lambda_1,\ldots,\lambda_m}$.
\end{lemma}

\begin{proof}
  Clearly, $\Psi^{\lambda_1,\ldots,\lambda_m}$ is a subsystem. We use induction on $m$ and first consider $m=1$ and $\lambda=\lambda_1$. Fix a set of simple roots $\Delta$ with positive system $\Psi^+$. By the proof of \cite[Theorem 10.3(a)]{humph}, there  is $w\in W$ such that $(\alpha,w(\lambda))\geq 0$ for all $\alpha\in  \Psi^+$. By Lemma \ref{lem:rts}
  this implies that the subsystem $\Psi^{w(\lambda)}$  has basis of simple roots $\Delta^{w(\lambda)} = \{ \alpha\in \Delta \mid (\alpha,w(\lambda))=0\}$.
  By the same lemma, the group $W^{w(\lambda)}$ is generated by all $s_\alpha$ with  $\alpha\in \Delta^{w(\lambda)}$. Now $\Psi^\lambda = w^{-1}(\Psi^{w(\lambda)})$ and
  $W^\lambda = w^{-1} W^{w(\lambda)} w$, so $W^\lambda$ is generated
  by all $w^{-1} s_{\alpha} w$ with $\alpha\in \Delta^{w(\lambda)}$. Since $w^{-1}s_{\alpha}w = s_{w^{-1}(\alpha)}$ and $w^{-1}(\Delta^{w(\lambda)})$ is a basis
  of $\Psi^\lambda$, it follows that $W^\lambda = W(\Psi^\lambda)$. Lastly,  observe that $\Psi^{\lambda_1,\ldots,\lambda_m} = (\Psi^{\lambda_1,\ldots,\lambda_{m-1}})^{\lambda_m}$ and $W^{\lambda_1,\ldots,\lambda_m} =  (W^{\lambda_1,\ldots,\lambda_{m-1}})^{\lambda_m}$, so the induction step follows by the same argument. 
\end{proof}

\subsection{The Weyl group $W(\Phi)^\theta$}\label{secWinWT} 
Recall that $W(\g,\h)=N_{G^\circ}(\h)/Z_{G^\circ}(\h)$ and that the full Weyl group $W(\Phi)\cong N_{G^c}(\h^c)/Z_{G^c}(\h^c)$ is generated by all reflections $s_\alpha$. The embedding $N_{G^\circ}(\h)\to  N_{G^c}(\h^c)$ induces an embedding of the real Weyl group into $W(\Phi)$, that is, we consider $W(\g,\h)$ as a subgroup
\begin{eqnarray*}W(\g,\h)\leq W(\Phi);
\end{eqnarray*}
see also \cite[(7.93) \& Proposition 7.19(c)]{knapp} or \cite[p.\ 950]{vogan}. In fact, it is also true that  $W(\g,\h)\leq W(\Phi)^\theta$: by the proof of \cite[Proposition~3]{dgrt}, the group $G=G^c(\R)$ is reductive in the sense of \cite{knapp}, and so $G^\circ$ is reductive by \cite[Proposition 7.19(f)]{knapp}; now \cite[(7.92b)]{knapp} shows that every element in $W(\g,\h)$ has a representative in $G^c$ that commutes with the Cartan involution $\theta$, which yields \[W(\g,\h)\leq W(\Phi)^\theta.\]

As a first step towards determining $W(\g,\h)$, we now describe the structure of  $W(\Phi)^\theta$. The following result comes from \cite[Proposition~3.12]{vogan}; we include a modified proof adapted to our set-up.

\begin{prop}\label{prop:3}
  We have $W(\Phi)^\theta = (W^{\mathrm{c}})^\theta \ltimes (W^{\mathrm{re}}\times W^{\mathrm{im}})$ with $W^\re,W^\imm\unlhd W(\Phi)^\theta$.
\end{prop}

\begin{proof}A small computation shows that if $\alpha\in \Phi^{\mathrm{re}}$ and $\beta\in
  \Phi$, then  $\theta(s_\alpha(\beta)) = s_\alpha(\theta(\beta))$, so that
  $W^{\mathrm{re}}\leq W(\Phi)^\theta$. If $\alpha\in \Phi^{\mathrm{re}}$ and $w\in W(\Phi)^\theta$, then $\theta(w(\alpha))  = w(\theta(\alpha))=w(-\alpha) = -w(\alpha)$, so  $w(\alpha) \in
  \Phi^{\mathrm{re}}$. Now from $ws_\alpha w^{-1}= s_{w(\alpha)}$ it follows that
  $W^{\mathrm{re}}\unlhd W(\Phi)^\theta$. The argument for $W^{\mathrm{im}}\unlhd W(\Phi)^\theta$ is similar. Let $\Phi^{\re,+}$ and $\Phi^{\imm,+}$ be positive systems for $\Phi^\re$ and $\Phi^\imm$, respectively.
  If  $w\in W(\Phi)^\theta$, then $w(\Phi^{\mathrm{re}}) = \Phi^{\mathrm{re}}$
  and $w(\Phi^{\mathrm{im}}) = \Phi^{\mathrm{im}}$, and so $w^{-1} (\Phi^{\mathrm{re},+})$ and
  $w^{-1} (\Phi^{\mathrm{im},+})$ are positive systems of $\Phi^{\mathrm{re}}$ and 
  $\Phi^{\mathrm{im}}$, respectively. Therefore 
  $w^{-1} (\Phi^{\mathrm{re},+}) = \mu^{-1}(\Phi^{\mathrm{re},+})$ and  $w^{-1} (\Phi^{\mathrm{im},+}) = \nu^{-1}(\Phi^{\mathrm{im},+})$ for some $\mu\in W^{\mathrm{re}}$ and
  $\nu\in W^{\mathrm{im}}$, respectively, see \cite[Section 10.3]{humph}. Lemma \ref{lemForm} shows that $(\Phi^{\mathrm{re}},\Phi^{\mathrm{im}})=0$, which implies  that $\mu$ is the identity on $\Phi^{\mathrm{im}}$ and   $\nu$ is the identity on $\Phi^{\mathrm{re}}$. It follows that $\Phi^{\mathrm{re},+} = w\mu^{-1}\nu^{-1} (\Phi^{\mathrm{re},+})$ and $\Phi^{\mathrm{im},+} = w\mu^{-1}\nu^{-1} (\Phi^{\mathrm{im},+})$. Setting $w_1 = w\mu^{-1}\nu^{-1}$, we have $w_1(\rre) = \rre$ and  $w_1(\rim)=\rim$.
  As $\Phi^{\mathrm{c}} = \{ \alpha\in \Phi \mid (\alpha,\rre)=(\alpha,\rim)=0\}$, Lemma~\ref{lem:2} shows that $w_1\in W^{\mathrm{c}}$. Since $\theta$ commutes with each of  $w,\mu,\nu$, it also commutes with $w_1$. Now $w=w_1\mu\nu$ shows that  $w\in (W^{\mathrm{c}})^\theta W^{\mathrm{re}} W^{\mathrm{im}}$.
  Every element of $(W^{\mathrm{c}})^\theta$ fixes $\rre$ and $\rim$, but no nontrivial element of
  $W^{\mathrm{re}}W^{\mathrm{im}}$ does that. Hence $(W^{\mathrm{c}})^\theta \cap W^{\mathrm{re}}
  W^{\mathrm{im}} = \{ 1\}$.
\end{proof}

\section{The real Weyl group}\label{secRWG}
\noindent We now look at the real Weyl group. As a preliminary step, in the next lemma we recall the  following
facts from \cite[Lemmas 5.1.4 \& 5.2.22]{gra16} and \cite[Lemma 6.1]{dfg}. Recall that we have chosen a Chevalley basis of $\g^c$ with elements $h_1,\ldots,h_\ell$ and $x_\alpha$ with $\alpha$ running over the root system $\Phi$ of $\g^c$.

\begin{lemma}\label{lemRaLa}
  Let $\g^c$ and $\g$ be as before.
  \begin{ithm}
\item If $w\in \g$ is nilpotent, then
  $t\mapsto \exp( t\ad w)$ for $t\in [0,1]$ is a path from the identity to
  $\exp(\ad w)$, so   $\exp( \ad w)\in G^\circ$.  The element $\exp( t\ad x_\alpha) \exp(-t^{-1}\ad x_{-\alpha}) \exp( t\ad x_\alpha)$
  lies in $N_{G^c}(\h^c)$ and maps to the reflection $s_\alpha$ in the
  Weyl group $W(\Phi)=N_{G^c}(\h^c)/Z_{G^c}(\h^c)$, independently of $t\in \C^*$.
 \item   Let $\alpha\in\Phi$; there exist
  $\lambda_\alpha,r_\alpha\in\C$ such that the following hold. First,  $\theta(x_\alpha)=\lambda_\alpha x_{\alpha\circ\theta}$ and $\lambda_\alpha^{-1}=\lambda_{-\alpha}=\lambda_{\alpha\circ\theta}$; moreover $\theta(h_\alpha)=h_{\alpha\circ\theta}$. Second,  $\sigma(x_\alpha)=r_\alpha x_{-\alpha\circ\theta}$ and
$r_\alpha^{-1}=r_{-\alpha}=\overline{r_{-\alpha\circ\theta}}$ (complex
 conjugate); moreover $\sigma(h_\alpha)=h_{-\alpha\circ\theta}$. Third,  $r_\alpha \lambda_\alpha$ is real.
  \end{ithm} 
\end{lemma}
 
The following lemma exhibits some subgroups of $W(\g,\h)$.

\begin{lemma} We have  $W^{\mathrm{re}},(W^{\mathrm{c}})^\theta,W^{\imm,\mathrm{c}}\leq W(\g,\h)$.
\end{lemma}
  \begin{proof}We freely use Lemma \ref{lemRaLa} in this proof; $r_\alpha$ and $\lambda_\alpha$ are defined as in that lemma.
    \begin{iprf}
      \item If $\alpha\in \Phi^{\mathrm{re}}$, then $\sigma(x_\alpha) = r_\alpha x_\alpha$,
        $\sigma(x_{-\alpha}) = r_\alpha^{-1} x_{-\alpha}$, and $\sigma(h_\alpha) = h_\alpha$. If $\sigma(x_\alpha)=-x_\alpha$, then set $\mu = \imath$, otherwise set $\mu = 1+r_\alpha$, so that  $\sigma(\mu x_{\alpha}) = \mu x_\alpha$ in either case. It follows that the element $\exp( \mu \ad x_\alpha) \exp( -\mu^{-1}\ad y_\alpha) \exp( \mu \ad x_\alpha)$
lies in $G^\circ$ and maps to $s_\alpha$ in $W(\Phi)$, so that $s_\alpha\in W(\g,\h)$.

\item Recall that $(W^{\mathrm{c}})^\theta$ is generated by elements of the form $s_\alpha s_{\theta(\alpha)}$, where $\alpha$ and $\theta(\alpha)$ lie in different
orthogonal components of $\Phi^{\mathrm{c}}$, see Lemma \ref{lem:1}; in particular, $x_{\pm \alpha}$ commutes with  $x_{\pm \theta(\alpha)}$. If $s_\alpha s_{\theta(\alpha)}$ is a generator of $(W^{\mathrm{c}})^\theta$, then both  $x_\alpha +r_\alpha x_{-\theta(\alpha)}$ and $x_{-\alpha}+r_{-\alpha}
x_{\theta(\alpha)}$ lie in $\g$ as they are invariant under $\sigma$. Being sums of two commuting nilpotent elements, they are nilpotent, so 
$$\exp( \ad(x_\alpha +r_\alpha x_{-\theta(\alpha)}) ) \exp( -\ad
(x_{-\alpha}+r_{-\alpha}x_{\theta(\alpha)})) \exp( \ad(x_\alpha +r_\alpha
x_{-\theta(\alpha)}) )$$ 
lies in $G^\circ$. Since the $x_{\pm \alpha}$ commute with
$x_{\pm \theta(\alpha)}$, this element is equal to
$$\exp( \ad x_\alpha) \exp(-\ad x_{-\alpha}) \exp(\ad x_\alpha)
\exp( r_\alpha \ad x_{-\theta(\alpha)})\exp( -r_{-\alpha} \ad x_{\theta(\alpha)} )
\exp( r_\alpha \ad x_{-\theta(\alpha)}),$$
and since $r_{-\alpha} = r_{\alpha}^{-1}$, the latter element maps to
$s_\alpha s_{\theta(\alpha)}$ in $W(\Phi)$, that is, $s_\alpha s_{\theta(\alpha)}\in W(\g,\h)$.

\item If $\alpha\in\Phi^{\imm}$, then  $\theta(\alpha)=\alpha$; moreover,  $\sigma(x_\alpha) = r_\alpha x_{-\alpha}$, $\sigma(x_{-\alpha}) = r_\alpha^{-1} x_\alpha$, and $\sigma(h_\alpha) =-h_\alpha$  with $r_\alpha\in\R$. Furthermore, $\theta(x_\alpha) = \lambda_\alpha x_\alpha$ and
because $\theta$ is an involution it follows that $\lambda_\alpha =\pm 1$. In the following suppose that $\lambda_\alpha=1$, so $\alpha$ is compact imaginar. Since $\lambda_{-\alpha} = \lambda_\alpha^{-1}$, it follows that  $-\alpha$ is compact imaginary if $\alpha$ is.

Since $\alpha$ is compact,  $u=\imath h_\alpha$, $x=x_\alpha +r_\alpha x_{-\alpha}$, and $y=\imath (x_\alpha-r_\alpha x_{-\alpha})$ lie in $\fk$; moreover, $[u,x]=2y$, $[u,y]=-2x$, and $[x,y] = -2r_\alpha u$, so they span a subalgebra $\mathfrak{a}$ of $\mathfrak{k}$. If $r_\alpha >0$,
then $\ad_{\mathfrak{a}} x$ has real eigenvalues and $\mathfrak{a}$ is
isomorphic to $\sl(2,\R)$, which is impossible because $\mathfrak{k}$ is
compact. This implies that $r_\alpha <0$, so $\xi = \sqrt{-1/r_\alpha}$ is real. We define  $a=\xi x$, $b=\xi y$, and $c=u$, so that $[a,b]=2c$, $[a,c]=-2b$, and  $[b,c]=2a$. Now we consider the group $\SL(2,\C)$ and its real Lie subgroup
$$\mathrm{SU}(2) = \left\{ \left(\begin{smallmatrix} v & -\bar w \\
  w & \bar v \end{smallmatrix}\right) \mid v,w \in \C, |v|^2+|w|^2=1\right\}.$$
Both are simply connected; for $\SL(2,\C)$ this is well-known, for $\mathrm{SU}(2)$ see  \cite[Proposition 1.15]{hall_lie}. The Lie algebra $\mathfrak{su}(2)$ of $\SU(2)$ has basis elements
$$A=\left(\begin{smallmatrix} 0 & \imath \\ \imath & 0 \end{smallmatrix}\right),~
B=\left(\begin{smallmatrix} 0 & -1 \\ 1 & 0 \end{smallmatrix}\right),~
C=\left(\begin{smallmatrix} \imath & 0 \\ 0 & -\imath \end{smallmatrix}\right)$$
satisfying $[A,B]=2C$, $[A,C]=-2B$, and $[B,C]=2A$. Mapping $(A,B,C)$ to $(a,b,c)$  yields an isomorphism $\phi\colon \mathfrak{su}(2) \to \mathfrak{a}$ which we extend to $\phi^c \colon \sl(2,\C) \to \mathfrak{a}^c$. Setting $h=-iC$, $e=(1/2\imath) (A-\imath B)$, and $f=(1/2\imath)(A+\imath B)$,
we have that $h,e,f$ is an $\sl_2$-triple; moreover, 
\[e = \left(\begin{smallmatrix} 0 & 1 \\ 0 & 0 \end{smallmatrix}\right),\quad  f=\left(\begin{smallmatrix} 0 & 0 \\ 1 & 0 \end{smallmatrix}\right),\quad \phi^c(e) = (\xi/\imath) x_\alpha,\quad \phi^c(f) = (\xi r_\alpha/\imath) x_{-\alpha}.\]
Because of simply-connectedness, $\phi$ and $\phi^c$ lift to unique Lie group
homomorphisms $F \colon \SU(2)\to G$ and $F^c \colon \SL(2,\C)\to G^c$, respectively, with
$F(\exp(z)) = \exp( \ad \phi(z) )$ for $z\in \mathfrak{su}(2)$ and $F^c(\exp(z)) = \exp( \ad \phi^c(z))$ for $z\in \sl(2,\C)$, see \cite[Theorem 5.6]{hall_lie}. Because of uniqueness,  $F$ is the restriction
of $F^c$ to $\mathfrak{su}(2)$.  Now consider \[M = \left(\begin{smallmatrix} 0 & 1 \\ -1 & 0 \end{smallmatrix}\right)\in\SU(2).\] Since $M=\exp(e)\exp(-f)\exp(e)$, we have $F^c(M) = \exp( \tfrac{\xi}{\imath} \ad x_\alpha )\exp( -\tfrac{\imath}{\xi} \ad x_{-\alpha} )\exp( \tfrac{\xi}{\imath} \ad x_\alpha )$, hence $F^c(M)$ is a representative in $G^c$ of $s_\alpha$. By what is said above, $F^c(M)=F(M)\in G$, hence $s_\alpha\in W(\g,\h)$. We note that the $\SL(2)$ argument is motivated by \cite[Section 5]{dc2} and \cite[p.\ 950]{vogan}.
    \end{iprf}
\end{proof} 

Together with the results of Section \ref{secWinWT}, we have
\[ (W^{\mathrm{c}})^\theta \ltimes (W^{\mathrm{re}}\times W^{\imm,\mathrm{c}})\leq W(\g,\h) \leq (W^{\mathrm{c}})^\theta \ltimes (W^{\mathrm{re}}\times W^{\imm});\]it remains to determine $W(\g,\h)\cap W^{\imm}$. Recall from the beginning of Section \ref{secWinWT} that every element in $W(\g,\h)$ has a representative in $G$ that commutes with $\theta$, hence $W(\g,\h)$ preserves the compact imaginary roots $\Phi^{\mathrm{im},\mathrm{c}}$ and so
$W^{\mathrm{im},\mathrm{c}} \subset (W(\g,\h) \cap W^{\mathrm{im}}) \subset W^{\mathrm{im},2}$ where $$W^{\mathrm{im},2} = \{ w\in W^{\mathrm{im}} \mid w(\Phi^{\mathrm{im},\mathrm{c}}) =\Phi^{\mathrm{im},\mathrm{c}}\}.$$ The next sections show that $W^{\imm,2}=Q\ltimes W^{\imm,\mathrm{c}}$ for some elementary abelian 2-group $Q$, and that $W(\g,\h)\cap W^{\imm}=A\ltimes W^{\imm,\mathrm{c}}$ for some subgroup $A\leq Q$. Determining $A$ is the difficult part.

\subsection{Superorthogonal roots} A subset $A=\{\alpha_1,\ldots,\alpha_m\}\subset \Psi$ of a root system is {\em strongly  orthogonal} if $\alpha_i \pm \alpha_j \not \in \Psi \cup \{ 0\}$ for
$i\neq j$; it is {\em superorthogonal} if the only roots in the span of
$A$ are $\pm \alpha_i$ and these are all distinct. Note that superorthogonal implies strongly orthogonal. Furthermore, if $A$ is strongly orthogonal, then $(\alpha_i,\alpha_j) = 0$ for $i\neq j$: indeed, if $(\alpha_i,\alpha_j)<0$, then $\alpha_i+\alpha_j$ is a root, and if $(\alpha_i,\alpha_j)>0$, then  $\alpha_i-\alpha_j$ is a root, see \cite[Lemma 9.4]{humph}. So strongly orthogonal implies orthogonal,
as it should. The next lemma is motivated by \cite[Lemma 3.19]{vogan}.

\begin{lemma}\label{lem:orth}
  Let $\Delta$ be a simple system of $\Psi$. If  $\alpha_1,\ldots,\alpha_m\in
  \Delta$ are orthogonal, then $\{\alpha_1,\ldots,\alpha_m\}$ is superorthogonal. 
\end{lemma}

\begin{proof}
  Let $\beta$ be a root in the span of the $\alpha_i$. If $\beta\ne\pm \alpha_i$ for all $i$, then we can write $\beta=\alpha_{i_1}+\cdots +\alpha_{i_r}$ such that all partial sums
  $\alpha_{i_1}+\cdots +\alpha_{i_k}$ are roots, see \cite[Corollary 10.2]{humph}. It follows that  $\alpha_i+\alpha_j\in \Psi$ for some $i,j$. If  $\alpha_i-r\alpha_j,\ldots, \alpha_i+q\alpha_j$ is the $\alpha_j$-string
  through $\alpha_i$, then $r-q = \langle \alpha_i,\alpha_j^\vee\rangle=0$, see \cite[p.\ 45]{humph}. Since $\alpha_i-\alpha_j$ is not a root by \cite[Lemma 10.1]{humph}, we have $r=0$; this forces $q=0$, a contradiction to $\alpha_i+\alpha_j\in\Psi$. This shows  that $\beta\in\{\pm\alpha_i\}$ for some $i$.
\end{proof}
Choose a positive system $\Phi^{\mathrm{im}, \mathrm{c}, +}$ in $\Phi^{\mathrm{im}, \mathrm{c}}$ to define $\rho^{\mathrm{im},\mathrm{c}} = \rho( \Phi^{\mathrm{im}, \mathrm{c}} )$ and set \begin{eqnarray}\label{eqQ} Q = \{ w\in W^{\mathrm{im},2} \mid w(\rho^{\mathrm{im},\mathrm{c}}) = \rho^{\mathrm{im},\mathrm{c}}\}.
\end{eqnarray}
  Choose a positive system $\Phi^{\mathrm{im},+}$ of $\Phi^{\mathrm{im}}$ such that $\rho^{\imm,\mathrm{c}}$ is dominant with respect to it, that is, such
  that $(\alpha,\rho^{\mathrm{im},\mathrm{c}}) \geq 0$ for all $\alpha\in \Phi^{\mathrm{im},+}$; this can be done as follows: let $V$ be the real span of
$\Phi^{\mathrm{im}}$; choose a basis of $V$ consisting of elements of
$\Phi^{\mathrm{im}}$; for $u,v\in V$ set $u<v$ if $(u,\rho^{\mathrm{im},\mathrm{c}}) <
(v,\rho^{\mathrm{im},\mathrm{c}})$
or if $(u,\rho^{\mathrm{im},\mathrm{c}})= (v,\rho^{\mathrm{im},\mathrm{c}})$ and the
first coefficient of $v-u$ with respect to the chosen basis of $V$ is positive;
then this is a root order and the corresponding positive system has the required properties. Consider the subsystem
$$\Phi^{\mathrm{im},\rho} = \{ \alpha\in \Phi^{\mathrm{im}} \mid
(\alpha,\rho^{\mathrm{im},\mathrm{c}}) = 0\}.$$By Lemma \ref{lem:rts}, the simple roots in $\Phi^{\mathrm{im},+}$
lying in $\Phi^{\mathrm{im},\rho}$ form a basis of the latter. We denote the set of these simple roots by $B=\{\alpha_1,\ldots,\alpha_m\}$. The next result is due to Knapp \cite[Proposition 3.20]{vogan}.

\begin{prop}\label{propB}Using the previous notation, the following hold.
\begin{ithm}
  \item $B=\{\alpha_1,\ldots,\alpha_m\}\subset\Phi^{\imm,\mathrm{nc}}$ is superorthogonal,
  \item $W^{\mathrm{im},2} = Q\ltimes W^{\mathrm{im},\mathrm{c}}$,
  \item  $Q = W^{\mathrm{im},2}\cap W(\Phi^{\imm,\rho})$. 
\end{ithm}
\end{prop}
   
\begin{proof}
  \begin{iprf}
    \item Lemma \ref{lemTriv} implies that  $\Phi^{\mathrm{im},\rho}\subset \Phi^{\mathrm{im},\mathrm{nc}}$. Since the $\alpha_i$ lie in a basis of $\Phi^{\mathrm{im},+}$, it follows that
  $(\alpha_i,\alpha_j)\leq 0$ for all $i,j$, see \cite[Lemma 10.1]{humph}. Suppose $(\alpha_i,\alpha_j)<0$ for  some $i,j$, so that $\alpha_i+\alpha_j\in \Phi^{\mathrm{im}}$ by  \cite[Corollary~9.4]{humph}. It is obvious that $\alpha_i+\alpha_j\in\Phi^{\imm,\rho}$, so that   $\alpha_i+\alpha_j \in \Phi^{\mathrm{im},\mathrm{nc}}$. But since $\alpha_i$ and   $\alpha_j$ also lie in  $\Phi^{\mathrm{im},\mathrm{nc}}$, a small computation shows that  $\alpha_i+\alpha_j\in  \Phi^{\mathrm{im},\mathrm{c}}$, a contradiction. We conclude that $(\alpha_i,\alpha_j)=0$ for all $i,j$, and now Lemma~\ref{lem:orth} proves that $B$ is superorthogonal. 
\item Note that $Q=\{ w\in W^{\mathrm{im},2} \mid w(\Phi^{\mathrm{im}, c, +}) = \Phi^{\mathrm{im}, c, +} \}$. 
 If $w\in W^{\mathrm{im},2}$, then $w^{-1}(\Phi^{\mathrm{im}, c, +})$ is a positive
  system for $\Phi^{\mathrm{im}, \mathrm{c}}$, so $w^{-1}(\Phi^{\mathrm{im}, c, +}) = \mu^{-1}(\Phi^{\mathrm{im}, c, +})$ for some $\mu\in W^{\mathrm{im},\mathrm{c}}$. This implies   $w\mu^{-1} =q\in Q$, and so $w=q\mu\in QW^{\mathrm{im},\mathrm{c}}$. By \cite[Theorem~10.3(e)]{humph}, any element of
  $W^{\mathrm{im},\mathrm{c}}$ that maps $\Phi^{\mathrm{im}, c, +}$ to itself is the identity, hence $Q\cap W^{\mathrm{im},\mathrm{c}} = 1$. If  $\alpha\in \Phi^{\mathrm{im},\mathrm{c}}$ and $w\in W^{\mathrm{im},2}$, then
  $ws_\alpha w^{-1} = s_{w(\alpha)}$ with $w(\alpha)\in \Phi^{\mathrm{im}, \mathrm{c}}$, so  $W^{\mathrm{im},\mathrm{c}}\unlhd W^{\mathrm{im},2}$.
\item  Lemma \ref{lem:2} shows 
 $W(\Phi^{\imm,\rho}) = \{w \in W^\imm \mid w(\rho^{\imm,\mathrm{c}}) = \rho^{\imm,\mathrm{c}}\}$, so $Q=W^{\imm,2}\cap W(\Phi^{\imm,\rho})$. 
  \end{iprf}
\end{proof}

\subsection{Constructing the last subgroup} Recall that $W^{\imm,\mathrm{c}}\leq W(\g,\h)\leq W^{\imm,2}=Q\ltimes W^{\imm,\mathrm{c}}$, so it remains to determine the subgroup $A\leq Q$ such that $W(\g,\h)\cap W^{\mathrm{im}}=A\ltimes W^{\imm,\mathrm{c}}$. Together with the previous results, this then leads to the decomposition of $W(\g,\h)$ as in \eqref{eqWDEC}. Recall that $Q=W(\Phi^{\imm,\rho})\cap W^{\imm,2}$, so \[A=Q\cap W(\g,\h)\leq W(\Phi^{\imm,\rho})\cap W(\g,\h).\] The next proposition describes $W(\Phi^{\imm,\rho})\cap W(\g,\h)$; intersecting with $Q$ gives $A$.

Let $P$ be the root lattice defined by the root system $\Phi$, that is, $P$ is the set of all integral linear combinations of the elements in $\Phi$. Since $\theta$ acts on $\Phi$, it also acts on $P$; let  $P^\theta$ be the sublattice of all $\mu\in P$ with $\mu=\theta(\mu)=\mu\circ \theta$.  Recall that $\langle \alpha,\beta^\vee\rangle=2(\alpha,\beta)/(\beta,\beta)$ for  $\alpha,\beta\in P$. The next result gives an explicit construction of $W(\Phi^{\imm,\rho})\cap W(\g,\h)$; this proposition is similar to \cite[Corollary~6.10]{dc2}, cf.\ \cite[Section 13]{adams}, but here we adapt it to our situation and give an independent proof.

\begin{prop}\label{propWI}
  Let
  $B=\{\alpha_1,\ldots,\alpha_m\}$ be the superorthogonal set in Proposition \ref{propB}. Then
\[ W(\Phi^{\imm,\rho})\cap W(\g,\h)=\{s_{\alpha_1}^{\epsilon_1} \cdots  s_{\alpha_m}^{\epsilon_m}\mid (\epsilon_1,\ldots,\epsilon_m)\in E\}\]
where $E = \{ (\epsilon_1,\ldots,\epsilon_m)\in \{0,1\}^m \mid
  \sum\nolimits_{i=1}^m \epsilon_i \langle \mu, \alpha_i^\vee\rangle = 0 \bmod 2
  \text{ for all }\mu\in P^\theta\}.$
\end{prop}

\begin{proof}The proof uses some standard results for semisimple algebraic groups; a classical reference is Steinberg's lecture notes \cite{stein}. Here we give detailed references to \cite{gra16}.  

First, we use the description in \cite[Section 5.2.3 and p.\ 182]{gra16} to construct the simply connected algebraic group $\widetilde{G}^c$  over $\C$ with
  Lie algebra isomorphic to $\g^c$: let $V^c$ be
  a $\g^c$-module such that the weights of $V^c$ generate the entire weight
  lattice $P$; let $\phi \colon \g^c \to \gl(V^c)$ be the corresponding faithful representation.
  Then $\widetilde{G}^c$ is the subgroup of $\GL(V^c)$ generated
  by all $\exp \phi(x) $ with $x\in \g^c$ nilpotent, see \cite[Section~5.2.3 and p.~172]{gra16}, with Lie algebra $\phi(\g^c)$. By \cite[Corollary 5.2.32]{gra16}, there is a surjective morphism of algebraic groups $\pi \colon \widetilde{G}^c \to  G^c$ with  $\ker\pi\leq Z(\widetilde{G}^c)$. If $g=\exp\phi(y)$ with $x\in\g^c$ and nilpotent $y\in\g^c$, then   $$g\phi(x)g^{-1}=(\exp \phi(y)) \phi(x) (\exp(-\phi(y)))= (\exp \ad \phi(y)) (\phi(x)),$$see \cite[Lemma 2.3.1]{gra16}; since $\phi^{-1}((\exp \ad \phi(y)) (\phi(x)))=(\exp \ad y) (x)$, we have  $\pi(g)= \exp \ad y$. In conclusion, \[\pi(g)(x) = \phi^{-1} (g\phi(x)g^{-1})=(\exp\ad y)(x).\]
By  \cite[Theorem 35.3(c)]{hum}, the group of real points $\widetilde{G}=\widetilde{G}^c(\R)$ is connected (in the
  Euclidean topology),  and $\pi(\widetilde{G})=G^\circ$ follows from \cite[7.4]{borelf}. Since $\widetilde{G^c}$ is simply connected, $\theta\colon \g^c\to \g^c$ lifts to a unique involution of $\widetilde{G}^c$ with
  differential equal to $\theta$, see  \cite[Theorem~5.6]{hall_lie}; by abuse of notation, we denote this involution by $\theta$. That theorem also shows   $\theta( \exp \phi(x) ) = \exp \phi(\theta(x))$ for all nilpotent $x\in \g^c$. To simplify notation, in the following we identify $\g^c$ with $\phi(\g^c)$, that is, we assume that $\phi$ is the identity. Now  define\[ N_{\widetilde{G}}(\h) = \{ g\in \widetilde{G} \mid \pi(g)(\h)=\h\}\quad\text{and}\quad Z_{\widetilde{G}}(\h) = \{ g\in \widetilde{G} \mid \pi(g)(x)=x \text{ for all } x \in \h\}.\]
 Since $\ker\pi\leq Z(\widetilde{G}^c)$ and $g=\exp y$ with $y\in\g^c$ nilpotent acts as $\pi(g)=\exp\ad y$ on $\g^c$, we have that $\pi$ maps $N_{\widetilde{G}}(\h)$ and $Z_{\widetilde{G}}(\h)$ onto $N_{G^\circ}(\h)$ and $Z_{G^\circ}(\h)$, respectively. Since $W(\g,\h)=N_{G^\circ}(\h)/Z_{G^\circ}(\h)$ by definition, we deduce $N_{\widetilde{G}}(\h)/Z_{\widetilde{G}}(\h) \cong W(\g,\h)$, so every element in $W(\g,\h)$ has a representative in $\widetilde{G}$. Recall that $\widetilde H^c=Z_{\widetilde{G}^c}(\h^c)$ is the connected torus in $\widetilde{G}^c$ with Lie algebra $\h^c$; we now recall some well-known facts for $$\widetilde K^c = \{ g\in \widetilde{G}^c \mid \theta(g)=g\}.$$
\begin{ithm}
\item[(1)]  $Z_{\widetilde K^c}(\h^c)=\widetilde H^c\cap \widetilde K^c$: Since $\widetilde H^c$ is connected and $\widetilde H^c\leq \widetilde{G}^c$ are matrix groups, \cite[Lemma 4.7.3]{gra16} yields $Z_{\widetilde{G}^c}(\widetilde H^c)=Z_{\widetilde{G}^c}(\h^c)$, hence $Z_{\widetilde K^c}(\h^c)=Z_{\widetilde K^c}(\widetilde H^c)=\widetilde K^c\cap \widetilde H^c$.
\item[(2)] $Z_{\widetilde{G}}(\h)=\widetilde H^c(\R)$: Since $\widetilde H^c=Z_{\widetilde{G}^c}(\widetilde H^c)=Z_{\widetilde{G}^c}(\h^c)$ and $\widetilde{G}=\widetilde{G}^c(\R)$, we have
\[Z_{\widetilde{G}}(\h^c)=\widetilde{G}\cap Z_{\widetilde{G}^c}(\h^c)=\widetilde{G}\cap Z_{\widetilde{G}^c}(\widetilde H^c)=Z_{\widetilde{G}}(\widetilde H^c)=\widetilde{G}^c(\R)\cap \widetilde H^c=\widetilde H^c(\R).\]Since $\widetilde H^c(\R)\subset \widetilde H^c$ is dense, $Z_{\widetilde{G}}(\widetilde H^c(\R))= Z_{\widetilde{G}}(\widetilde H^c)$. Now  $Z_{\widetilde{G}}(\h)=Z_{\widetilde{G}}(\h^c)$ proves the claim.
\item[(3)] $N_{U^c}(\h^c)=N_{U^c}(\widetilde H^c)$ for any subgroup $U^c\leq \widetilde{G}^c$: Recall that $\widetilde{G}^c$ and $\g^c$ are matrix structures. Let $g\in \widetilde{G}^c$ and consider the regular map $\Ad(g)\colon\widetilde{G}^c\to\widetilde{G}^c$, $x\mapsto gxg^{-1}$. If $g$ normalises $\widetilde H^c$, then the differential of $\Ad(g)$ is $\g^c\to \g^c$, $x\mapsto gxg^{-1}$, see \cite[p.\ 110]{gra16}, and maps $\h^c$ to $\h^c$, hence $g$ normalises $\h^c$. Conversely, suppose $g\h^cg^{-1}=\h^c$. Note that $H_0^c=\Ad(g)(\widetilde H^c)$ is an algebraic subgroup and the differential of $\Ad(g)$ maps $\h^c$ to the Lie algebra $\h_0^c$ of $H_0^c$. By assumption, $g\h^cg^{-1}=\h^c$, which implies $\h_0^c=\h^c$, hence $H_0^c=\widetilde H^c$ by \cite[Theorem 4.3.3]{gra16}.
\item[(4)]  $N_{\widetilde{G}}(\widetilde H^c(\R))=N_{\widetilde{G}}(\h)$: Using (3), we see that  $g\in\widetilde{G}$ normalises $\widetilde H^c(\R)$ if and only if it normalises $\widetilde H^c$, if and only if it normalises $\h^c$, if and only if it normalises $\h$.
\end{ithm} 
By \cite[Proposition 6.3.2]{adams2}, there is an isomorphism $N_{\widetilde{G}}(\widetilde H^c(\R))/\widetilde H^c(\R)\cong N_{\widetilde K^c}(\widetilde H^c)/(\widetilde H^c\cap \widetilde K^c)$. Together with the above observations, this implies that 
\begin{eqnarray*}W(\g,\h)\cong N_{\widetilde{G}}(\h)/Z_{\widetilde{G}}(\h)\cong N_{\widetilde K^c}(\h^c)/Z_{\widetilde K^c}(\h^c).
\end{eqnarray*} 
Recall that $\pi$ maps $\widetilde G$ onto $G^\circ$. Since $\h^c$ is $\theta$-stable, $\theta$ is an involution on $\widetilde H^c$, and we can decompose $\widetilde H^c=\widetilde H^c_+ \widetilde H^c_-$,
  where $\widetilde H^c_\pm= \{ h\in \widetilde H^c \mid \theta(h) = h^{\pm 1}\}$. For each $\alpha_i$ in the superorthogonal set $B$ define \[g_i = \exp( x_{\alpha_i} ) \exp( -x_{-\alpha_i} )\exp( x_{\alpha_i} ),\]
  so that  $g_i\in \widetilde{G}^c$ maps to $s_{\alpha_i}$ in the full Weyl group $W(\Phi)$, see Lemma \ref{lemRaLa}.  Since each $\alpha_i$ is noncompact imaginary, $\theta(x_{\alpha_i})=-x_{\alpha_i}$. Since $\theta( \exp x ) = \exp \theta(x)$ for nilpotent $x\in \g^c$, this shows $\theta(g_i)=g_i^{-1}$; moreover, the $g_i$ commute because of the superorthogonality. 

Since $W(\Phi)=N_{\widetilde{G}^c}(\h^c)/Z_{\widetilde{G}^c}(\h^c)$, the set of elements in $N_{\widetilde{G}^c}(\h^c)$ that map to a product  $s_{\alpha_{i_1}}\cdots s_{\alpha_{i_r}}$ in $W(\Phi)$ is exactly $gZ_{\widetilde{G}^c}(\h^c)$ where  $g=g_{i_1}\cdots g_{i_r}\in \widetilde{G}^c$. Recall that $W(\g,\h)$ is a subgroup of $W(\Phi)$ and we have $W(\g,\h)=N_{\widetilde K^c}(\h^c)/Z_{\widetilde K^c}(\h^c)$. This shows that the image of the above $g$ lies in $W(\g,\h)$ if and only if there exists $h\in Z_{\widetilde G^c}(\h^c)$ such that $gh\in \widetilde K^c$, that is, $\theta(gh)=gh$. Writing $h=h_+h_-$ with $h_\pm \in \widetilde H^c_\pm$, we have that  $\theta(gh) = \theta(g)h_+h_-^{-1}$, and therefore  $\theta(gh) = gh$ if and only if
  $g^{-1}\theta(g)=  h_-^{2}$. By what is said above,   $\theta(g)=g^{-1}$, so $g$ maps to an element of $W(\g,\h)$ if and only if
  there is $h_-\in \widetilde H_-^c$ with $g^{-2} = h_-^{2}$. Thus, we  study the diagonalisable group $\widetilde H^c_-$. 

Recall that $P$ is the weight lattice
  of $\Phi$. Since $\widetilde{G}^c$ is simply connected, \cite[Proposition 5.3.12]{gra16} shows that we can identify the character group of $\widetilde H^c$ with $P$.  More precisely, let $\gamma_1,\ldots,  \gamma_m$ be the simple roots of $\Phi$. Following \cite[pp.\ 162 \& 172]{gra16}, for a root $\alpha\in\Phi$ and $t\in \C$ write $x_\alpha(t) = \exp( tx_\alpha )$, and for $t\in \C^*$ define
  $w_\alpha(t) = x_\alpha(t)x_{-\alpha}(-t^{-1})x_\alpha(t)$ and $h_\alpha(t) = w_\alpha(t)w_\alpha(1)^{-1}.$ By \cite[Example~5.2.33]{gra16}, the elements of $\widetilde H^c$ can uniquely be written as
  $h_{\gamma_1}(t_1)\cdots h_{\gamma_m}(t_m)$, and  \cite[Lemma~5.2.17]{gra16} shows that   $\lambda(h_\alpha(t)) = t^{\langle \lambda,\alpha^\vee\rangle}$ for all $\lambda\in P$. Since $\widetilde H_-^c$ is an algebraic subgroup of $\widetilde H^c$,  it is the intersection of
  the kernels of the characters in a sublattice of $P$, see \cite[Proposition~3.9.5]{gra16}. This sublattice is $(1+\theta)P = \{ \lambda+  \theta(\lambda) \mid \lambda\in P\}$: if $h\in \widetilde H_-^c$, then  $(\lambda+\theta(\lambda))(h)  = \lambda(h)\lambda(h^{-1})=1$ for all $\lambda\in P$;  if $h\in \widetilde H^c$ satisfies $\mu(h)=1$ for all $\mu\in (1+\theta)P$, then $\lambda(h\theta(h))=1$ for all   $\lambda\in P$; since $P$ contains all roots, this implies $h\theta(h)=1$, so $h\in \widetilde H_-^c$. Thus $\widetilde H_-^c$ is defined by $(1+\theta)P$. We now show that \[P^\theta/(1+\theta)P\cong \widetilde H_-^c/(\widetilde H_-^c)^\circ,\] see also \cite[Proposition 12.3]{chev}. Clearly,  $(1+\theta)P\leq P^\theta$. Note that both  lattices have
  the same rank equal to the dimension $e$ of the $1$-eigenspace of $\theta$  on the $\mathbb{Q}$-space spanned by $P$. It is obvious that  $P^\theta$ is pure, meaning that $P/P^\theta$ is torsion free. Together, $P^\theta$ is the
 the smallest sublattice containing $(1+\theta)P$ that is pure, that is, $P^\theta$ is the purification of $(1+\theta)P$. Since $P^\theta$ is pure,  \cite[Proposition~3.9.6]{gra16} shows that it defines a connected algebraic subgroup of $\widetilde H_-^c$, namely an $e$-dimensional torus $E$. Since $P^\theta$ and $(1+\theta)P$ have the same rank, it follows that $(1+\theta)P$ defines a group isomorphic to $F\times E$, where $F$ is finite. Thus,  $P^\theta$ defines $E=(\widetilde H_-^c)^\circ$. Moreover, it follows from \cite[Remark 3.9.8]{gra16} that $F$ is isomorphic to the torsion subgroup of $P/(1+\theta)P$; so $P^\theta/(1+\theta)P\cong F \cong  \widetilde H^c_-/(\widetilde H^c_-)^\circ$.

 Note that if  $\mu\in P^\theta$, then $\mu+\mu\in (1+\theta)P$, which shows that $P^\theta/(1+\theta)P$ is an elementary abelian 2-group. We conclude that $(\widetilde H_-^c)^2 = (\widetilde H^c_-)^\circ$. Putting things together, we see that $g=g_{i_1}\cdots g_{i_r}\in \widetilde{G}^c$ as above maps to an element of $W(\g,\h)$ if and only if $g^{-2}$ lies in $(\widetilde H_-^c)^2=(\widetilde H_-^c)^\circ$, if and only if $\mu(g^{-2})=1$ for all $\mu\in P^\theta$. From the definitions of $g_i$ and $h_{\alpha_i}(t)$, it follows that $g_i^{-2} = h_{\alpha_i}(-1)$. Furthermore, as mentioned above, for $\mu\in P$ we have
  $\mu(h_{\alpha_i}(-1)) = (-1)^{\langle \mu,\alpha_i^\vee\rangle}$. It follows that $s_{\alpha_{i_1}}\cdots s_{\alpha_{i_r}}$ lies in $W(\g,\h)$ if and only if $\sum\nolimits_{j=1}^r \langle \mu, \alpha_{i_j}^\vee \rangle $ is even for all $\mu\in P^\theta$. 
\end{proof}

\subsection{Description of the real Weyl group}\label{secFINAL}
We can now state the main result; for convenience, we recall the assumptions. Let $\g^c$ be a semisimple complex Lie algebra with real form $\g$. Let $\theta$ be a Cartan involution of $\g$ and fix a $\theta$-stable Cartan subalgebra $\h$ of $\g$.  Let $G^c$ be the adjoint group of $\g^c$, let $G$ be the group consisting of all $g\in G^c$ with $g(\g)=\g$, and set  $W(\g,\h)=N_{G^\circ}(\h)/Z_{G^\circ}(\h)$.

\begin{theorem}\label{thmFINAL}We have $W(\g,\h)= (W^{\mathrm{c}})^\theta \ltimes (W^{\mathrm{re}}\times (A\ltimes W^{\imm,\mathrm{c}}))$ with $A=Q\cap W(\Phi^{\imm,\rho})\cap W^{\imm,2}$.
\end{theorem}

The groups $W^{\mathrm{re}}$, $W^{\imm,\mathrm{c}}$, and $W^{\mathrm{c}}$ are described in Section \ref{secRSS}; the group $(W^{\mathrm{c}})^\theta$ is the set of fixed points under the action of $\theta$. The group $Q$ is described in \eqref{eqQ} and the group $W(\Phi^{\imm,\rho})\cap W^{\imm,2}$ is described in Proposition \ref{propWI}. Together, this yields a construction of $W(\g,\h)$ based on the root system of $\g^c$.

\section{Computations}\label{secComp} 
An implementation of our construction algorithm for $W(\g,\h)$ is distributed with the software package CoReLG for the computer algebra system GAP \cite{gap}. Based on the description provided here, the implementation of all steps of the algorithm is straightforward
so we will not comment on this. As an example, we have constructed the real Weyl groups for all simple real forms of rank at most 8. Table \ref{tab1} lists some results where for each Lie type the computation took the most time (in seconds), or where $A$ was maximal (size 16). We have split the timing into the time for computing the root system and Chevalley basis (column labeled time RS) and the remaining time for computing the real Weyl group (column labeled time WG). The column labeled id lists the id of the real form as given by the CoReLG function {\sf IdRealForm} and the column labeled csa lists the position of the respective Cartan subalgebra in the output of the CoReLG function {\sf CartanSubalgebrasOfRealForm}.

{\small
\begin{table}[h]
\begin{tabular}{ccc||rrrr|rr}
$\g$ & id & csa & $|(W^{\mathrm{c}})^\theta|$ & $|W^\re|$ & $|A|$ & $|W^{\imm,\mathrm{c}}|$ & time RS & time WG\\\hline
$E_8(8)$ &  [E,8,2] &8 & 2 & 4 & 4 & 576 &  4.2 & 3.7\\
$E_7(7)$  & [E,7,2] &8 & 2 & 4 & 4 & 16 &1.8 & 1.3\\ 
$E_6(6)$  & [E,6,2] &4& 6 & 2 & 8 & 1 & 0.5& 0.3\\
$\mathfrak{sl}(9,\R)$  &[A,8,6] &5& 24 & 1 & 16 & 1 & 0.5&0.3\\
$\mathfrak{so}(8,9)$  &[B,8,5] &9 & 2 & 192 & 8 & 1 & 1.0 & 0.7\\
$\mathfrak{sp}(8,\R)$ & [C,8,6] &21 & 2 & 4 & 4 & 24 & 1.0&0.8\\
$\mathfrak{so}(3,13)$ &[D,8,5] & 14& 2 & 4 & 4 & 16 & 1.0 & 0.5\\
$F_4(4)$  &[F,4,2] & 3& 1 & 2 & 2 & 6 & 0.2&0.1\\
$G_2(2)$& [G,2,2] & 3& 1 & 2 & 2 & 1 & 0&0\\\cdashline{1-9}
$\mathfrak{so}(8,9)$  & [B,8,5]& 12& 6 & 16 & 16 & 1 & 1.0 & 0.6\\
$\mathfrak{sp}(8,\R)$  &[C,8,6]& 15& 24 & 16 & 16 & 1 & 1.0&0.9\\
$E_8(8)$  &[E,8,2]& 5& 24 & 16 & 16 & 1 & 5.1 & 3.5\\[1ex] 
\end{tabular}\caption{Examples of some real Weyl group orders}\label{tab1}
\end{table}

\end{document}